\documentclass[11pt]{article}
\usepackage[french]{babel}
\usepackage{latexsym}
\usepackage{amsmath,amsxtra,amssymb,latexsym, amscd,amsthm,}
 \usepackage[pdftex]{color,graphicx}
 \usepackage[all]{xy}
\usepackage{pgf,tikz}
\usepackage[T1]{fontenc}
\usepackage{cleveref}
\usepackage{amsfonts}
\usetikzlibrary{arrows}
\usepackage{graphicx}
\usepackage{caption}
\usepackage{subcaption}
\usepackage{graphpap}
\usepackage{rotating}
\usepackage{bm}
\theoremstyle{plain}

\newtheorem*{tHm}{Theorem}
\newtheorem*{Thm}{Main Theorem}
\newtheorem{lemma}{Lemma}[section]
\newtheorem{question}{Question}
\newtheorem{prop}{Proposition}

\theoremstyle{definition}

\newtheorem{Prop}{Proposition}[section]
\newtheorem{definition}{Definition}[section]
\newtheorem{Definition}{Definition}
\newtheorem{cor}{Corollary}[section]

\newcommand{\z}{\mathbb{Z}}

\newcommand{\sg}{\Sigma}
\newcommand{\gm}{\Gamma}

\title{\sc{On dual unit balls of Thurston norms}}
\author{Abdoul Karim SANE}
\date{ \small{}}

\begin{document}
\renewcommand{\proofname}{Proof}
\renewcommand{\abstractname}{Abstract}
\renewcommand{\refname}{Bibliography}
\maketitle
\begin{abstract}
Thurston norms are invariants of 3-manifolds defined on their second homology vector spaces, and understanding the shape of their dual unit ball is a (widely) open problem. W. Thurston showed that every symmetric polygon in $\z^2$, whose vertices satisfy a parity property, is the dual unit ball of a Thurston norm on a 3-manifold. However, it is not known if the parity property  on the vertices of polytopes is a sufficient condition in higher dimension or if their are polytopes, with mod ~2 congruent vertices, that cannot be realized as  dual unit balls of Thurston norms. In this article, we provide a family of polytopes in ~$\z^{2g}$ that can be realized as dual unit balls of Thurston norms on 3-manifolds. These polytopes come from intersection norms on oriented closed surfaces and this article widens the bridge between these two norms.  
\end{abstract}
\begin{section}{Introduction} Given an oriented 3-manifold $M$ with tori boundaries (we consider only these manifolds), W. Thurston \cite{thurst} defined  a semi-norm on the second homology vector space of $M$. Let $a\in H_2(M,\z)$ be an integer class, then $a$ admits representatives that are  disjoint unions of properly embedded surfaces $S_i$ in ~$M$. The Thurston norm of $a$ is given by:
$$x(a)=\min_{[\cup_iS_i]=a}\{\displaystyle{\sum_i{\max\{0, -\chi(S_i)}}\}\}.$$

If $M$ is prime ---any embedded sphere in $M$ bounds a ball--- and atoroidal ---has no essential torus---, then $x$ extends to a norm on $H_2(M,\mathbb{R})$. By construction $x$ takes integer values in $H_2(M,\z)$. It is an \textit{\textbf{integer norm}}: a norm on a vector space ($H_2(M,\mathbb{R})$ in this case) that takes integer values on a top dimensional lattice ($H_2(M,\mathbb{Z})$ in this case). W. Thurston showed that the dual unit ball of an integer norm on a vector space $E$ (relatively to a lattice ~$\Lambda$) is the convex hull of finitely many 1-forms $u_i\in E^*$ that take integer values on $\Lambda$:
$$x(a)=\max_{u_i}\{\langle u_i,a\rangle\}.$$
For a given 3-manifold with tori boundaries, the vectors defining the dual unit ball of the Thurston norm satisfy the \textit{\textbf{parity property}}: $$\forall (i,j);\hspace{0.2cm} u_i\equiv u_j~\mod 2.$$

\noindent When a surface $S$ realizes the norm in its homology class, namely when $x([S])=-\chi(S)$, $S$ is said to be \textit{\textbf{minimizing}}. Thurston gave some conditions under which a given surface is minimizing. For instance, closed leaves of transversally oriented foliations without Reeb component are minimizing. D. Gabai \cite{Gab} showed the converse, namely that minimizing surfaces are leaves of foliations without Reeb component on $M$.  W. Thurston also showed that if an embedded surface $S$ in ~$M$ is a fiber of a surface bundle over the circle, then $S$ is minimizing and there is a unique vector $u_i \in H^2(M,\z)$ such that $x([S])=~\langle u_i,S\rangle$. Moreover, if $S'$ is another surface such that ~$x([S'])=~\langle u_i,S'\rangle$, then $S'$ is the fiber of a surface bundle. Therefore, homology classes of fiber of bundles over the circle belong to an open cone on finitely many faces, called \textit{\textbf{fibered faces}}, of the unit ball of the Thurston norm on ~$M$. In this article, we are interested in the shape of these polytopes that appear as dual unit balls of Thurston norms. 

Just after defining his norm, Thurston started to compute it on a few examples of 3-manifolds. Even better, he showed the following:
\begin{tHm}[W. Thurston \cite{thurst}, Thm 6]
Every symmetric polygon in ~$\z^2$ with vertices defined by $\mod 2$ congruent vectors  is the dual unit ball of the Thurston norm of a 3-manifold.  
\end{tHm}

Our main result is an extension of Thurston's theorem to symmetric polytopes in dimension $2g$. To begin with, let us introduce intersection norms on  an oriented closed surface ~$\sg_g$. Intersection norms were quickly introduced by V. Turaev \cite{turaev}, and received a novel interpretation in the article of M. Cossarini and P. Dehornoy \cite{Direc}.  

Let $\gm=\{\gamma_1,...,\gamma_n\}$ be a finite collection of closed curves on $\sg_g$ in generic position. If $\gm$ is a \textit{\textbf{filling collection}}, that is its complement on $\sg_g$ is a union of topological disks, the intersection norm $N_{\gm}$ associated to $\gm$ restricted to homology with integer coefficients is the function defined by:
\begin{align*}
N_{\Gamma}:H_1(\sg_g,\mathbb{Z})&\longrightarrow \mathbb{Z}\\
                     a&\longmapsto \inf\{\mathrm{card}\{\alpha\cap\gm\}; [\alpha]=a \}.
\end{align*}
Intersection norms are also integer norms. Therefore, the dual unit ball of ~$N_{\gm}$ is the convex hull of finitely many vectors $v_i\in H^1(\sg,\z)$. As for Thurston norms, the vectors $v_i$ satisfy the parity property, since geometric intersection and algebraic intersection between curves have  the same parity. 

We recall that the norm is completely determined by the vectors $v_i$: $$N_{\gm}(a)=\displaystyle{\max_{v_i}\{\langle v_i,a\rangle\}}.$$ 

M. Cossarini and P. Dehornoy \cite{Direc} gave an algorithm for the computation of all the vectors defining an intersection norm.  

\begin{Definition}
A filling collection $\gm$ is \textit{\textbf{homologicaly  nontrivial}} if there exists an orientation $\overset{\rightarrow}{\gm}$ of $\gm$ such that $[\overset{\rightarrow}{\gm}]$ is a nontrivial  homology class.

A \textit{\textbf{homologicaly  nontrivial polytope}} in $\z^{2g}$ is a symmetric polytope  with vertices given by mod 2 congruent vectors that appears as the dual unit ball of an intersection norm on $\sg_g$ associated to a homologicaly  nontrivial  collection.    
\end{Definition}

If a filling collection $\gm$ is not homologicaly  nontrivial, the norm $N_{\gm}$ is an even fonction in $H_1(\sg_g,Z)$; that is the coordinates of the vectors of the dual unit ball are all even. So, many dual unit ball of intersection norms are homologicaly  nontrivial.

Our main theorem brings Thurston norms closer to intersection norms.

\begin{Thm} If $P$ is a homologicaly nontrivial  polytope, then it is the dual unit ball of a Thurston norm on a 3-manifold. 
\end{Thm}

Unlike intersection norms, computing the dual unit ball of a Thurston norm is difficult. There is an algorithm (\cite{Agol}, \cite{Marc}) that determines whether a given surface ~$S$ is Thurston norm minimizing or not. This algorithm uses the theory of sutured manifold hierarchies introduced by D. Gabai in \cite{Gab}. He used hierarchies to construct taut foliations on the exterior of many knots and as a consequence to determine their genus. M. Scharlemann \cite{scharl} then showed that hierarchies determine the Thurston norm of a homology class in general. The difficulty in this algorithm is to find a sutured manifold hierarchy for checking that an embedded surface is minimizing or not. The Thurston norm minimizing problem is NP-complete \cite{Marc}.

 Since the matter is less complicated for intersection norms, our main theorem provides many polytopes which are dual unit balls of Thurston norms. Nonetheless, we showed in \cite{Element} that there are symmetric polytopes in $\z^4$ that are not dual unit balls of intersection norms. This result makes the characterization of polytopes (in even dimensions) that appear for those two norms widely open; and one wonders if those polytopes that are not dual unit balls of intersection norms are also not dual unit balls of Thurston norms.

The proof of our main theorem goes through a detailed analysis of incompressible surfaces in the exterior of a knot in a circle bundle over a surface. From that analysis, we obtain a total description of minimizing surfaces. This avoids many of the foliation technicalities to check if a given surface is minimizing, and also the use of sutured manifold hierarchies algorithm.

\begin{Definition}
Let $M$ be a 3-manifold. An embedded surface $S$ in $M$ is \textit{\textbf{incompressible}} if any simple curve on $S$ which bounds a disk in $M$ also bounds a disk in $S$. It is equivalent to say that $i_*:\pi_1(S)\longrightarrow~\pi_1(M)$ is injective. 
\end{Definition}

If $S$ is not incompressible in $M$, then one can cut $S$ along a compressing disk. This cutting operation reduces the complexity of the surface. Therefore, a minimizing surface for the Thurston norm is incompressible. 

\begin{tHm}[F. Waldhausen \cite{waldh}]
Let $\pi:M\longrightarrow\sg_g$ be a circle bundle. Then, up to isotopy, an incompressible surface $S$ in $M$ is either \textbf{\textit{vertical}}, that is $\pi^{-1}(\pi(S))=S$, or \textit{\textbf{horizontal}}, that is $\pi|_S:S\longrightarrow\sg_g$ is a finite covering.  
\end{tHm}

For $\pi:M\longrightarrow\sg_g$ a circle bundle with Euler number 1 and $K$ an oriented knot in ~$M$ such that $\pi(K)$ is a nontrivial  homology class, we denote by $M_K$ the complement in $M$ of a tubular neighborhood of $K$.
\begin{Definition}
Let $S$ be a surface embedded in $M_K$. The \textit{\textbf{closure}} of $S$ in ~$M$, denoted by $\bar{S}$, is the surface embedded in $M$ obtained by capping all the boundary components of $S$ by disks.
\end{Definition}

\noindent The closure $\bar{S}$ is embedded in $M$ and $S=\bar{S}\cap M_K$.\vspace{0.2cm} 

 For the proof of Main Theorem, we show the following:
\begin{prop}
Let $S$ be an incompressible surface in $M_K$ and $\bar{S}$ its closure in $M$. There is sequence of incompressible surfaces $S_0\longrightarrow S_1\longrightarrow...\longrightarrow S_n=\bar{S}$ such that:
\begin{itemize}
\item $S_0$ is a disjoint union of vertical surfaces;
\item $S_{i+1}$ is obtained by attaching a handle to $S_i$. 
 \end{itemize}
\end{prop}

\begin{paragraph}{Outline of this article:}
 Section ~\ref{sec1} recalls Thurston's construction of polygons as dual unit balls of Thurston norm of 3-manifolds. In Section \ref{sec2} we extend Waldhausen's classification of a knot complement in a circle bundle with nonzero Euler number. Section ~\ref{sec3} is devoted to the proof of the Main Theorem. 
\end{paragraph}
\end{section}

\begin{section}{Thurston's construction of 3-manifolds realizing polygons}\label{sec1}
In this section, we review Thurston's proof \cite{thurst} of the fact that symmetric polygons in $\z^2$ with vertices represented by $\mod2$ congruent vectors are dual unit balls of Thurston norms. It will be helpful for understanding our generalization.  

Let $\gm=\{\gamma_1,...,\gamma_n\}$ be a filling collection of closed geodesics  on a flat torus ~$\mathbb{T}^2$. Since any component of $\gm$ is simple and non-separating, there is an orientation of each component of $\gm$ such that the oriented collection ~$\overrightarrow{\gm}$ is  nontrivial  in homology (every collection of geodesics on the torus is homologicaly  nontrivial). 

\begin{figure}[htbp]
\begin{center}
\includegraphics[scale=0.11]{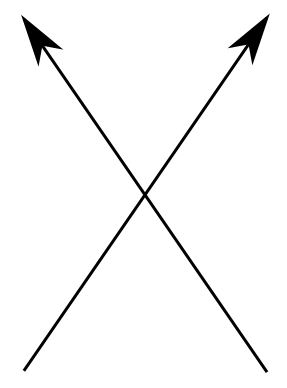}\hspace{2cm}
\includegraphics[scale=0.11]{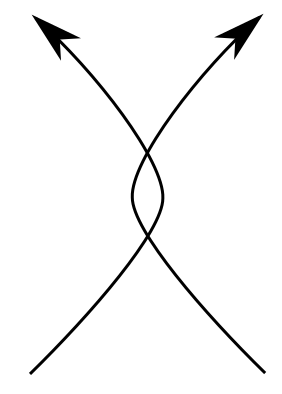}
\put(-27,5){\huge{$\longrightarrow$}}
\caption{Attaching two curves on a double point.}
\label{attach}
\end{center}
\end{figure}

By applying the operation depicted on Figure ~\ref{attach} at finitely many double points, we obtain a filling closed curve $\overrightarrow{\gamma}$ in $\mathbb{T}$ (which is not geodesic anymore). 

Now, let $\pi:M\longrightarrow\mathbb{T}$ be the circle bundle over $\mathbb{T}$ with Euler number $1$. Then, $H_2(M)$ is isomorphic to $H_1(\mathbb{T}$). 
\end{section}

Let $K$ be a lift of $\overrightarrow{\gamma}$ and $M_K$ be the complement in $M$ of a tubular neighborhood $T(K)$ of $K$. The morphism 
\begin{align*}
r:H_2(M)&\longrightarrow H_2(M_K,\partial M_K)\\
                     [S]&\longmapsto [S\cap M_K]
\end{align*}
is an isomorphism. In fact, we have the following exact sequence:  
\[...\rightarrow H_2(T(K))=0\rightarrow H_2(M)\rightarrow H_2(M,T(K))\rightarrow H_1(T(K))\rightarrow H_1(M)\rightarrow...\] 
Since $[\overrightarrow{\gamma}]=\pi_*(K)$ is nonzero, the inclusion $H_1(T(K))\rightarrow H_1(M)$ is injective. It follows that the map $H_2(M)\rightarrow H_2(M,T(K))$ is an isomorphism. By excision, we obtain the isomorphism  $r:H_2(M)\longrightarrow H_2(M_K,\partial M_K)$. 

Canonical representatives of $H_2(M_K,\partial M_K)$ are of the form $\pi^{-1}(\alpha)\cap M_K$, where $\alpha$ is an oriented simple curve in $\mathbb{T}$. Since $\pi^{-1}(\alpha)$ is a torus, then $$-\chi(\pi^{-1}(\alpha)\cap M_K)=|(\pi^{-1}(\alpha)\cap K)|=|\alpha\cap\gm|.$$

Thurston showed that if $\alpha$ minimally intersects $\gm$, so that $\pi^{-1}(\alpha)\cap M_K$ is minimizing: 
\begin{equation}\label{equ}
x([\pi^{-1}(\alpha)\cap M_K])=\displaystyle{\sum_i{i(\alpha,\gamma_i)}}.
\end{equation}

The technical part of Thurston's proof is the construction of a foliation on $N_K$ without Reeb component having $\pi^{-1}(\alpha)\cap M_K$ as a leaf. 

Equation (\ref{equ}) describes exactly an equality between Thurston norm on ~$N_K$ and the intersection norm on the torus associated to $\gm$. It can be rewritten as follows: $$x(a)=N_{\gm}(\pi_*(a)).$$ 

Now, we will extend Thurston's construction to higher genus surfaces using intersection norms, namely for every circle bundle $\pi:M\longrightarrow \sg_g$ of Euler number ~1. For the general case, there are essentially two differences:
\begin{itemize}
\item there exist filling collections that are not homologicaly  nontrivial  since the collection can be made of only separating closed simple curves,
\item there are examples  of filling collections $\gm$ and $N_{\gm}$-minimizing oriented curves $\alpha$ for which $\pi^{-1}(\alpha)\cap N_K$ is not minimizing for the Thurston norm (see Figure \ref{canal}). 
\end{itemize}

\begin{figure}[htbp]
\begin{center}
\includegraphics[scale=0.17]{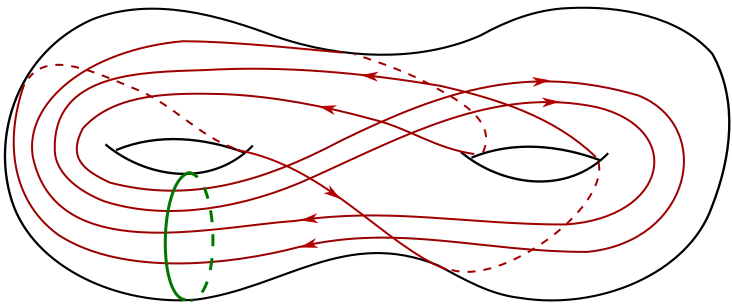}\hspace{1cm}
\includegraphics[scale=0.18]{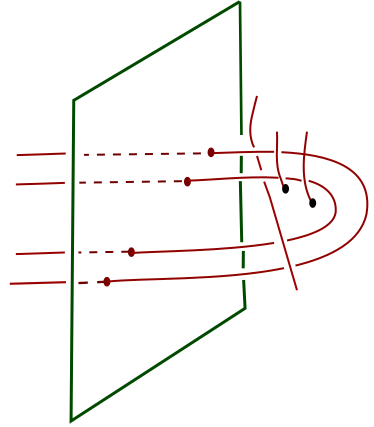}\hspace{1cm}
\includegraphics[scale=0.18]{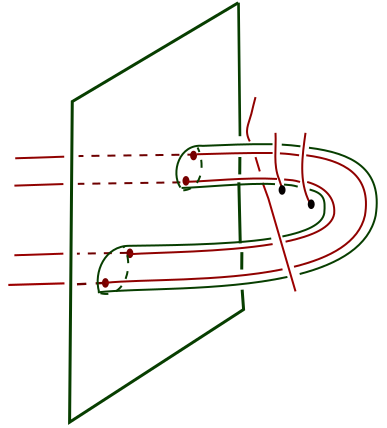}
\caption{The vertical surface $S$ (in the middle) over the green curve is a torus with four boundary components. By replacing these four boundary components by a handle, one obtained  a genus $2$ closed surface $S'$ (right picture) and $|\chi(S')|<|\chi(S)|$.}
\label{canal}
\end{center}
\end{figure}  

\begin{section}{Incompressible surfaces in nontrivial  knot complements in circle bundles.}\label{sec2}

 Given a $3$-manifold, a natural question is to classify incompressible surfaces up to isotopy. For the case of a circle bundle, a complete answer has been given by Waldhausen \cite{waldh}. Let us recall it in its general form. 
 
 Let $\pi:M\longrightarrow\sg_{g,k}$ be a circle bundle over a genus $g$ orbifold with $k$ cone points. We denote by ~$e(M)$ the Euler number of $M$. 
\begin{tHm}[F. Waldhausen \cite{waldh}]\label{waldh}
Let $S$ be an incompressible surface in $M$. Then, up to isotopy, $S$ is vertical, namely $\pi^{-1}(\pi(S))=S$, or horizontal, that is $\pi|_S:S\longrightarrow \sg_{g,k}$ is a branched cover. 
\end{tHm}
\noindent One can check the proof of Waldhausen's theorem in Hatcher's notes \cite{hatcher}.\vspace{0.2cm}

Vertical surfaces correspond to pre-images of essential simple closed curves in $\sg_{g,k}$ and a horizontal surface $S$ satisfies:
$$\chi(\sg_{g,k})-\frac{\chi(S)}{n}=\displaystyle{\sum_i{(1-\frac{1}{q_i})}};$$
where $n$ is the degree of the branched cover and the $q_i$ are the multiplicities of the fibers. When the base is a regular surface $\sg_g$, then $\chi(S)=n.\chi(\sg_g)$.

The existence of horizontal surfaces in $M$ depends on the Euler number. More precisely, a circle bundle admits a horizontal surface if and only if its Euler number is zero (\cite{hatcher}, Proposition 2.2). So, $H_2(M)$ is isomorphic to $H_1(\sg_g)$ when $e(M)\neq0$, with vertical surfaces as representatives of elements of $H_2(M)$.

Now, we push Waldhausen's classification a little bit further. Let $K$ be an oriented knot in $M$ ($e(M)\neq0$) that is nontrivial in homology, and $M_K:=M-\overset{\circ}{T}(K)$ be the exterior of $K$ in $M$. 

 \begin{Prop}(Proposition 1)\label{propp} 
 Let $S$ be an incompressible surface in $M_K$ and $\bar{S}$ its closure in $M$. There is sequence of incompressible surfaces $S_0\longrightarrow S_1\longrightarrow...\longrightarrow S_n=\bar{S}$ such that:
\begin{itemize}
\item $S_0$ is a disjoint union of vertical surfaces;
\item $S_{i+1}$ is obtained by attaching a handle to $S_i$. 
 \end{itemize}
 \end{Prop}
 \begin{proof}
 Let $\bar{S}$ be the closure of $S$ in $M$. If $\bar{S}$ is incompressible in $M$, then $\bar{S}$ is vertical and $S_0=S_n=\bar{S}$.
 
 If $\bar{S}$ is not incompressible, we obtain a sequence $\bar{S}\longrightarrow S_1\longrightarrow...\longrightarrow ~S_n$, where each step consists of cutting $S_i$ along an essential simple curve which bounds a disk in $M$, and taking the closure of the surface obtained. This process ends with a (possibly non connected) incompressible surface $S_n$ in $M$ which is a disjoint union of vertical surfaces. The reverse sequence achieves the proof. 
 \end{proof}
 
 The proposition above shows that the only obstruction for an incompressible surface to be vertical comes from attaching handles like in Figure ~\ref{canal}.
 
 \begin{definition} Let $S$ be an incompressible surface in $M_K$ and $\alpha$ an essential simple curve on $\bar{S}$ which bounds a disk $\mathbb{D}_{\alpha}$ in $M$. The \textit{\textbf{weight}} of $\alpha$ is the integer ~$w(\alpha)$ defined by:
 $$w(\alpha)=\min\{\rm{card}\{\mathbb{D}_{\alpha'}\cap K\}, \alpha' \hspace{0,1cm} \rm{isotopic\hspace{0,1cm} to}\hspace{0,1cm} \alpha\}$$ 
 
The \textit{\textbf{verticality defect}} of $S$ is the integer $vd(S)$ defined by:
$$vd(S)=\max_{\alpha}\{w(\alpha), \alpha \hspace{0,1cm} \rm{essential}\}.$$  
 \end{definition}\vspace{0,5cm}
 
 \noindent It is easy to see that if $S$ has verticality defect equal to zero, then $S$ is a vertical surface: $$S=\pi^{-1}(\alpha)\cap M_K,$$ where $\alpha$ is a simple closed curve on $\sg_g$. Moreover, if $vd(\bar{S})=1$, then $S$ is homologous to a vertical surface with the same Euler characteristic. In fact, if $\alpha$ is a simple curve on $\bar{S}$ such that $w(\alpha)=1$, we can cut $\bar{S}$ along $\alpha$ to obtain a surface ~$\bar{S}_1$. The surface $S_1:=\bar{S}_1\cap M_K$ has two more boundary components than $S$ and one handle less and is homologous to $S$. It follows that $\chi(S)=\chi(S_1)$. Repeating this process, we obtain a vertical surface $S_n$ with the same Euler characteristic as $S$.    
 
 We end this section with some definitions.\\
 Let $A$ and $B$ be two sub-arcs of $K$ such that $a:=\pi(A)$ and $b:=\pi(B)$ are simple arcs with extremities $\partial a=\{t,x\}$ and $\partial b=\{y,z\}$. Let $\lambda_1$ and $\lambda_2$ be two arcs from $t$ to $y$ and $x$ to $z$, respectively, such that $\lambda_1$, $a$, $\lambda_2$ and $b$ bound a topological disk.
 
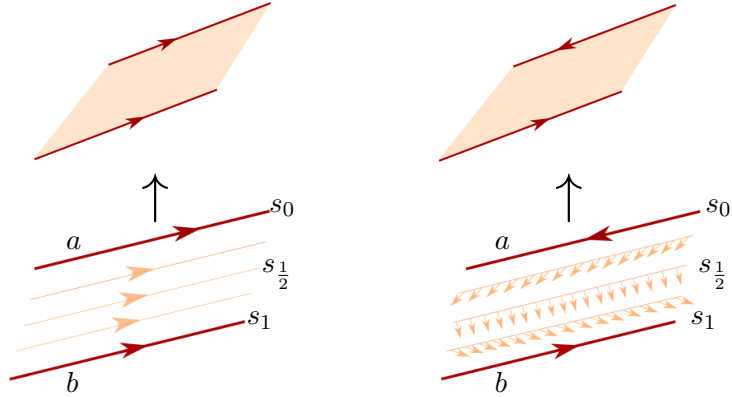
\begin{figure}[htbp]
\begin{center}
\includegraphics[scale=0.22]{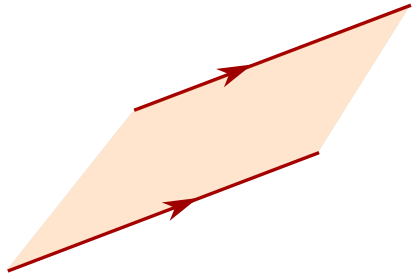}\hspace{2cm}
\includegraphics[scale=0.22]{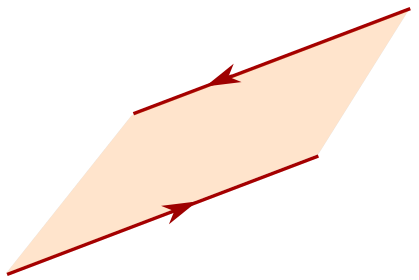}\\ \vspace{0,1cm}
\includegraphics[scale=0.3]{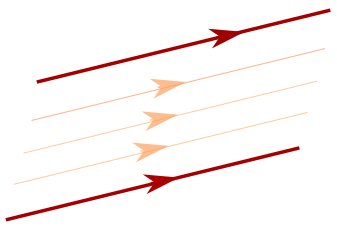}\hspace{2cm}
\includegraphics[scale=0.3]{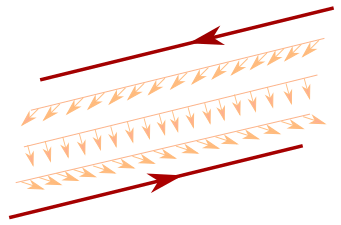}
\put(-75,23){\huge{$\uparrow$}}
\put(-20,23){\huge{$\uparrow$}}
\put(-58,23){$s_0$}
\put(-59,15){$s_{\frac{1}{2}}$}
\put(-61,8){$s_1$}
\put(0,23){$s_0$}
\put(-1,15){$s_{\frac{1}{2}}$}
\put(-2,8){$s_1$}
\put(-85,18){$a$}
\put(-85,-1){$b$}
\put(-28,18){$a$}
\put(-28,-1){$b$}
\caption{Rectangle between two arcs obtained by lifting a homotopy between two sections. On the left, we have the case where the orientations of the arcs agree and on the right we have the case where the orientations are opposite.}
\label{rec}
\end{center}
\end{figure}

  The arcs $a$ and $b$ can be seen as sections of the unit tangent bundle of their supports, and there is a homotopy (see Figure ~\ref{rec}) of sections $s_t$ such that:
 \begin{itemize}
 \item $\rm{pr}_1(s_t)$ is an isotopy between the support of $a$ and $b$, with extremities gliding in $\lambda_1$ and $\lambda_2$;
 \item $s_0=a$ and $s_1=b$.
   \end{itemize}  
The lift of $s_t$ in $M$ gives a rectangle $R$ from $A$ to $B$ and when we blow-up ~$R$, we obtain a handle enclosing $A$ and $B$.

\begin{lemma}
If $H$ is a handle in $M$ enclosing two sub-arcs $A$ and $B$ whose projections are simple arcs, then $H$ is isotopic to the blow-up of a rectangle between $A$ and $B$.  
\end{lemma}

\begin{proof}
Since $H$ is a compressible handle enclosing $A$ and $B$, then there is an isotopy between $A$ and $B$ inside $H$. This isotopy gives a rectangle $R$ between ~$A$ and $B$ and the blow-up of that rectangle is inside $H$. Therefore, ~$H$ is isotopic to the blow-up of $R$.  
\end{proof}

The construction described above works for more than two sub-arcs and in what follows, we will consider handles as blow-up of rectangles between sub-arcs.    

\end{section}

\begin{section}{Proof of the main theorem}\label{sec3}
  Let us start this section with the following statement: if two filling collections $\gm$ and $\gm'$ differ by an "attachment" (see Figure ~\ref{attach}), then the intersection norm associated to $\gm$ is equal to the one associated to $\gm'$ (see \cite{Element}). Therefore any intersection norm is realized by one filling curve $\gamma$, not necessarily in minimal position.
  
\begin{figure}[htbp]
\begin{center}
\includegraphics[scale=0.15]{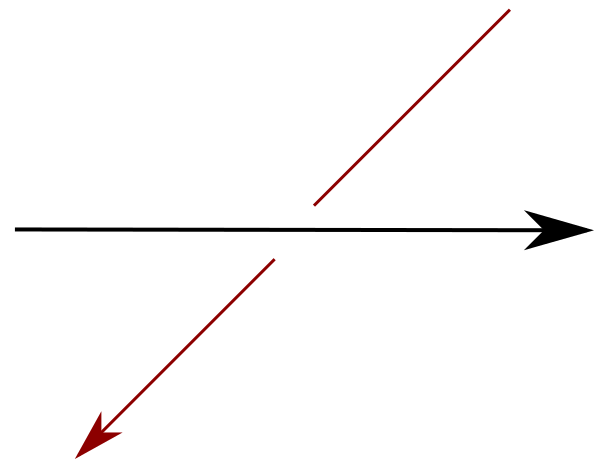}\hspace{3cm}
\includegraphics[scale=0.19]{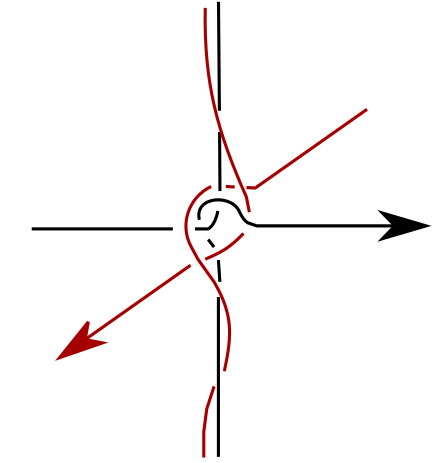} 
\put(-50,11){\huge{$\longrightarrow$}}
\put(-47,15){$(a)$}\\
\includegraphics[scale=0.2]{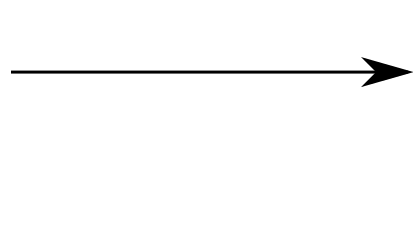}\hspace{3cm}
\includegraphics[scale=0.2]{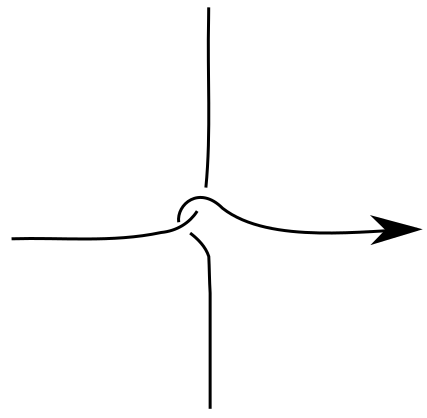}
\put(-50,11){\huge{$\longrightarrow$}}
\put(-47,15){$(b)$}
\caption{($a$) Modification of $K$ around a fiber of a double point of $\gamma$. Each arc (the dark and the red one) individually follows a fiber once, and is linked to itself ($(b)$ shows the modification of a single arc). Along the fiber, the modified arcs form a braid with two components.}
\label{palmier}
\end{center}
\end{figure}  
  
Let $\overrightarrow{\gamma}$ be an oriented filling curve with non vanishing class in homology. Let $\pi:M\longrightarrow\sg_g$ be a circle bundle with Euler number equal to $1$. As we have seen, ~$H_2(M)$ is isomorphic to $H_1(\sg_g)$. Instead of only taking a lift of $\overrightarrow{\gamma}$ in ~$N$, we add the modification depicted in Figure ~\ref{palmier}-$a$ on the neighborhood of the fiber of double points of $\gamma$. Let $\hat{K}$ be the knot obtained and $M_{\hat{K}}$ be the exterior of ~$\hat{K}$. Since ~$\pi(\hat{K})$ is still homologous to $\overrightarrow{\gamma}$, then $H_2(M_{\hat{K}})$ is isomorphic to ~$H_1(\sg_g)$ with vertical surfaces as canonical representatives.
     
As we have seen, Thurston's construction does not trivially extends to higher genus surfaces since a minimizing surface $S$ could have verticality defect greater than two. Our modification, which consists of braiding the knot ~$K$ along fibers (see Figure ~\ref{palmier}-$a$), increases the complexity of  incompressible surfaces with verticality defect greater than two.

\begin{definition}
Let $H_{\alpha}$ be a handle with $\partial H_{\alpha}=\{\alpha_1,\alpha_2\}$. Let $\lambda$ be any simple arc from $\alpha_1$ to $\alpha_2$.
The handle $H_{\alpha}$ is \textit{\textbf{horizontal}} if the homotopy class ---with fixed extremities--- of $\alpha$ in $M$ has no fibers. 
\end{definition}

\begin{lemma}\label{vd}
Let $S_1$ and $S_2$ be two vertical surfaces in $M_{\hat{K}}$ on which we attach a handle $H_{\alpha}$ to obtained a surface $S:=S_1\underset{H_{\alpha}}{\#}S_2$.

If $w(\alpha)\geq 2$, then there is a surface $S'$ homologous to $S$ such that $$-\chi(S')<-\chi(S).$$  
\end{lemma} 

\begin{proof}
If $\pi(H_{\alpha})$ does not contain a double point of $\overset{\rightarrow}{\gamma}$, then $S$ is compressible; the curve $\beta$ (Figure ~\ref{fig1}-a) which is obtained by summing to fibers in $S_1$ and ~$S_2$ along $H_{\alpha}$ is essential in $S$ and vanishes in $M_{\hat{K}}$. So, we can reduce the complexity of $S$ in this case.

\begin{figure}[htbp]
\begin{center}
\includegraphics[scale=0.15]{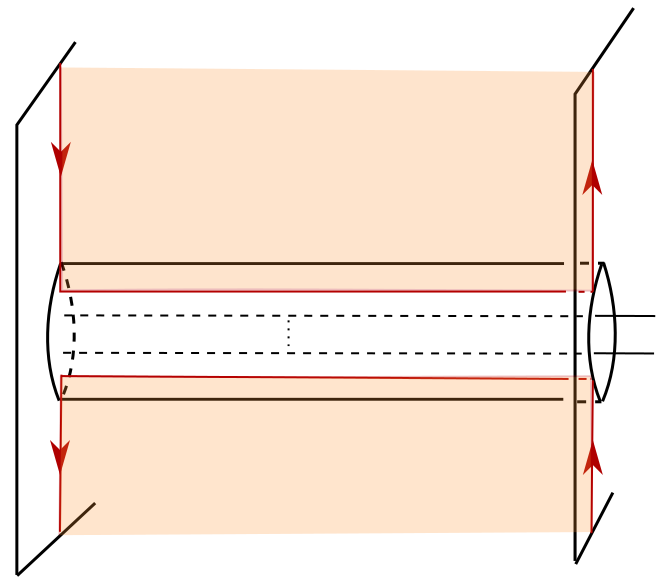}
\put(-3,19){\small{$\beta$}}
\put(-37,-2){\small{$S_1$}}
\put(-4,-2){\small{$S_2$}}
\put(-20,-3){\small{(a)}}\hspace{2cm}
\includegraphics[scale=0.20]{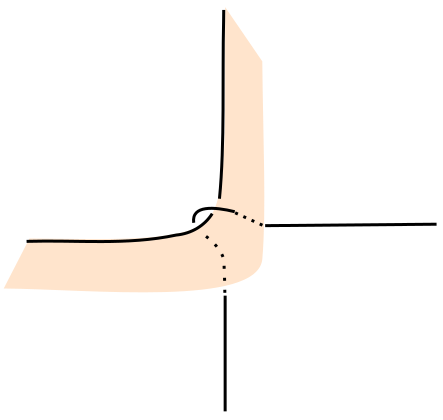}
\put(-18,-3){\small{(b)}}
\caption{(a) Compression disk in $M_{\hat{K}}$ bounded by an essential curve ~$\beta$ in ~$S$. (b) The arc around the fiber of a double point which intersects a rectangle. This shows that the rectangle cannot go completely along the fiber.}
\label{fig1}
\end{center}
\end{figure}
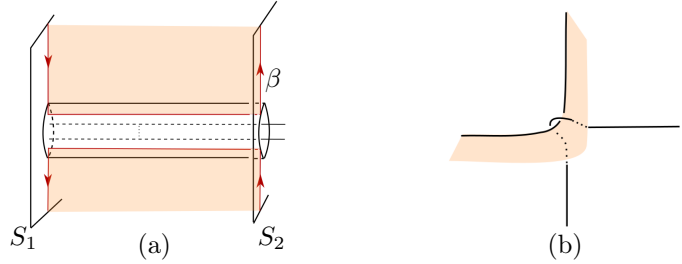

Now, suppose that  $\pi(H_{\alpha})$ contains double points of $\overset{\rightarrow}{\gamma}$. We claim that ~$H_{\alpha}$ is horizontal. Let us see $H_{\alpha}$ as the blowing up of a rectangle between sub-arcs of $\hat{K}$. Since a rectangle stays on one side of an arc, it follows that it cannot follow a sub-arc of $\hat{K}$ along a fiber (see Figure ~\ref{fig1}-b).

Finally, if $H_{\alpha}$ is horizontal and $\pi(H_{\alpha})$ contains a double point $p$, then the fiber ~$\pi^{-1}(p)$ intersects $H_{\alpha}$ twice. Therefore $H_{\alpha}$ intersects $\hat{K}$  four times the modification above $p$ and those four intersection points define four boundary components on $S$. By attaching a new handle along the fiber $\pi^{-1}(p)$ which encloses those four boundary components we obtain a surface $S'$ with one more handle and four boundary components less.
 So $-\chi(S')\leq-\chi(S)$.   
\end{proof}  

\begin{cor}\label{cor1}
Let $S$ be a surface embedded in $M_{\hat{K}}$. If $S$ is Thurston norm minimizing, then $vd(S)\leq1$.
\end{cor}
\begin{proof}
Since a minimizing surface $S$ is incompressible, by Proposition \ref{propp}, ~$S$ is obtained by attaching finitely many handle between vertical surfaces embedded in $M_{\hat{K}}$. By Lemma \ref{vd}, each handle has height less than or equal to ~$1$. It follows that $vd(S)\leq1$.    
\end{proof}

Now, we are able to prove the main theorem.
\begin{proof}[Proof of the main theorem]
Let $S$ be a (Thurston norm) minimizing surface in $M_{\hat{K}}$. By Corollary ~\ref{cor1}, $vd(S)\leq 1$. If $vd(S)=0$ then $S=\pi^{-1}(\alpha)\cap M_{\hat{K}}$. So $x(S)=N_{\gamma}(\alpha)$.

If $vd(S)=1$, then one can replace every handle of $S$ by two boundaries by cutting along essential simple curves in $S$ which are trivial in $M$. This operation does not increase the positive part of the Euler characteristic and we obtain at the end an incompressible surface $S'$ in the same homology class as $S$ and such that $vd(S')=0$. Again in this case, there is a vertical surface which minimizes the Thurston norm. 
So $x(a)=N_{\gm}({\pi_*(a)})$.  
\end{proof}

Homologicaly nontrivial polytopes  realize by our construction do not have fibered faces since a fibration of $M_{\hat{K}}$ by vertical surfaces would have given a foliation on $\sg_g$ without singularities.   

Our main theorem links the realization problems of intersection norms and Thurston norms. In \cite{Element}, we showed that any polytope ~$P$ in $\mathcal{P}_8$ (the set of non degenerate symmetric sub-polytopes of $[-1,1]^4$ with eight vertices) is not the dual unit ball of an intersection norm.

\begin{question}
Let $P\in\mathcal{P}_8$. Is $P$ the dual unit ball of an Thurston norm on a 3-manifold?
\end{question}

By Gabai theorem on the fact that minimizing surfaces are leaves of foliations without Reeb component, this question is somehow related to the studying of the topology of (Reebless) foliated 3-manifold with pairs of pants or one-holed torus as a leaves. 
\end{section}
\begin{paragraph}{Acknowledgements:} I am grateful toward Pierre Dehornoy and Jean-Claude Sikorav for introducing me to this subject. 
\end{paragraph}
\vspace{0,5cm}

Institut Fourier, Université Grenoble Alpes.\\
\textit{email: abdoul-karim.sane@univ-grenoble-alpes.fr}
\end{document}